\documentclass{article}

\usepackage{amsmath,amssymb,amsthm,enumitem,esint,graphicx,hyperref,mathtools,url,xcolor,zref-clever}
\usepackage{todonotes}

\setlist[enumerate,1]{label=(\roman*)}

\usepackage[style=alphabetic]{biblatex}
\numberwithin{equation}{section}

\newtheorem{theorem}{Theorem}[section]
\newtheorem{lemma}[theorem]{Lemma}
\newtheorem{proposition}[theorem]{Proposition}
\newtheorem{corollary}[theorem]{Corollary}

\theoremstyle{definition}
\newtheorem{definition}[theorem]{Definition}
\newtheorem{remark}[theorem]{Remark}

\zcRefTypeSetup{section}{cap}

\zcRefTypeSetup{figure}{cap}
\zcRefTypeSetup{item}{name-sg= ,name-pl= }
\zcRefTypeSetup{equation}{name-sg= ,name-pl= }

\zcRefTypeSetup{theorem}{cap}
\AddToHook{env/definition/begin}{%
\zcsetup{countertype={theorem=definition}}}
\zcRefTypeSetup{definition}{cap}
\AddToHook{env/example/begin}{%
\zcsetup{countertype={theorem=example}}}
\zcRefTypeSetup{example}{cap}
\AddToHook{env/lemma/begin}{%
\zcsetup{countertype={theorem=lemma}}}
\zcRefTypeSetup{lemma}{cap}
\AddToHook{env/corollary/begin}{%
\zcsetup{countertype={theorem=corollary}}}
\zcRefTypeSetup{corollary}{cap}
\AddToHook{env/proposition/begin}{%
\zcsetup{countertype={theorem=proposition}}}
\zcRefTypeSetup{proposition}{cap}
\AddToHook{env/remark/begin}{%
\zcsetup{countertype={theorem=remark}}}
\zcRefTypeSetup{remark}{cap}

\newcommand\f\frac
\newcommand\tx\text
\newcommand\intd{\mathop{}\!\mathrm{d}}
\newcommand\ind[1]{1_{#1}}

\newcommand\h{\mathcal H}
\newcommand\Ll{\mathcal L}
\newcommand\loc{\mathrm{loc}}
\newcommand\vint\fint
\newcommand{\mres}{\!\mathbin{\vrule height 1.6ex depth 0pt width
                0.13ex\vrule height 0.13ex depth 0pt width 1.1ex}\!}

\newcommand\BV{\mathrm{BV}}

\DeclareMathOperator\var{var}

\DeclareMathOperator\dist{dist}

\DeclareMathOperator\capa{Cap}

\newcommand\R{\mathbb R}
\newcommand\Z{\mathbb Z}
\newcommand\N{\mathbb N}
\newcommand\Sph{\mathbb S}
\newcommand\Om{\Omega}
\newcommand\eps{\varepsilon}
\newcommand{\ch}{\text{\raise 1.3pt \hbox{$\chi$}\kern-0.2pt}}
\newcommand\restrict[2]{#1|_{#2}}
\newcommand\restrictl[2]{\restrict{#1}{#2}}
\newcommand\restrictzero[2]{#1\,\mres#2}
\newcommand\restrictzerol[2]{\restrictzero{#1}{(#2)}}
\newcommand\borel[1]{\mathcal B(#1)}
\newcommand\borelm[2]{\mathcal B_{#1}(#2)}
\newcommand\meas[1]{\mathcal M(#1)}
\newcommand\measp[1]{\mathcal M^+(#1)}

\newcommand\B{\mathcal B}
\newcommand\C{\mathcal C}

\newcommand\id{\mathrm{Id}}

\newcommand\M{{\mathrm M}}
\newcommand\Ml[1]{\M_{<#1}}
\newcommand\Mg[1]{\M_{\geq#1}}

\newcommand\mb[1]{\mathop{\partial_*}{#1}}

\newcommand\lmo{\Ll}
\newcommand\lm[1]{\lmo(#1)}

\newcommand\smo{\h^{d-1}}
\newcommand\sm[1]{\smo(#1)}

\newcommand\lmlo{\Ll^{d-1}}
\newcommand\lml[1]{\lmlo(#1)}
\newcommand\lmio{\Ll^1}
\newcommand\lmi[1]{\lmio(#1)}

\newcommand\median[2]{\mathrm m(#1,#2)}

\newcommand\partop[2]{{#2}^{\mathrm{#1}}}
\newcommand\gsingular[1]{\partop s{#1}}

\newcommand\gac[1]{\partop a{#1}}

\newcommand\singular[1]{\partop sD{#1}}
\newcommand\cantor[1]{\partop cD{#1}}
\newcommand\ac[1]{\partop aD{#1}}
\newcommand\jump[1]{\partop jD{#1}}

\newcommand\sjump[2]{\partop jD_{#1}#2}

\makeatletter
\newcommand{\oset}[3][0ex]{%
  \mathrel{\mathop{#3}\limits^{
    \vbox to#1{\kern-2\ex@
    \hbox{$\scriptstyle#2$}\vss}}}}
\makeatother

\begin{document}

\title{The centered maximal operator removes the non-concave Cantor part from the gradient}

\author{%
Panu Lahti\footnote{
Academy of Mathematics and Systems Science,
Chinese Academy of Sciences,
Beijing 100190, PR China
\texttt{panulahti@amss.ac.cn}
}
\and
Julian Weigt\footnote{
University of Warwick,
Mathematics Institute,
Zeeman Building,
Coventry CV4 7AL,
United Kingdom,
\texttt{julian.weigt@warwick.ac.uk}
}
}

\maketitle

\begin{abstract}
We study regularity of the centered Hardy--Littlewood maximal function $\M f$ of a function
$f$ of bounded variation in $\R^d$, $d\in \N$.
In particular, we show that at $|\cantor f|$-a.e.\ point $x$ where $f$ has a non-concave blow-up, it holds that $\M f(x)>f^*(x)$.
We further deduce from this that if the variation measure of $f$ has no jump part and its Cantor part has non-concave blow-ups,
then BV regularity of $\M f$ can be upgraded to Sobolev regularity.
\end{abstract}

\begingroup
\begin{NoHyper}%
\renewcommand\thefootnote{}\footnotetext{%
}%
\addtocounter{footnote}{-1}%
\end{NoHyper}%
\endgroup

\section{Introduction}

The so-called
$W^{1,1}$-problem, posed in \cite{MR2041705}, asks whether the Hardy-Littlewood
maximal function of a Sobolev function $f\in W^{1,1}(\R^d)$,
defined (for nonnegative $f$) by
\[
\M f(x)
\coloneqq
\sup_{r>0}
\frac1{\lm{B(x,r)}}\int_{B(x,r)}f(y)\intd\lmo(y),\qquad x\in \R^d,
\]
is also locally in the class $W^{1,1}(\R^d)$, with
\[
\Vert \nabla \M f\Vert_{L^1(\R^d)}\le C\Vert \nabla f\Vert_{L^1(\R^d)}
.
\]
The problem arose as a natural extension of the case $1<p<\infty$,
where the analogous result is known to hold,
as first shown by Kinnunen \cite{Kin}.

Many papers, for example \cite{APL,AlPer09,MR3695894,MR2280193,MR1898539,weigt2024variation},
have considered this $W^{1,1}$-problem for the more general class of functions of bounded variation
$\BV(\R^d)\supset  W^{1,1}(\R^d)$, and for different maximal operators, e.g.\
the \emph{uncentered} maximal function,
see \cite{zbMATH07215904} for a survey.
In particular, Kurka \cite{Kur}
gave a positive answer to the $W^{1,1}$-problem for BV functions in one dimension.
Apart from BV or Sobolev regularity, one can also consider other regularity
properties of the maximal function, such as approximate differentiability \cite{MR2550181}
or quasicontinuity \cite{panubvsobolev}.

An addition to preserving regularity, \cite{APL} also addressed the question of whether maximal operators can even \emph{increase} regularity.
The authors proved that for $f\in\BV(\R)$, the gradient of its uncentered maximal function belongs to $L^1(\R)$, i.e.\ is not only a finite measure but an absolutely continuous one.
For the centered maximal operator this does not hold:
if a function has jumps, so can its centered maximal function.
In general, the variation measure of a BV functions $f$ can be divided into three parts:
\[
	Df=\ac f+ \cantor f+ \jump f,
\]
see \zcref{sec:preliminaries}, \zcref{eq:decomposition of variation measure}.
In \cite{Kur} Kurka also showed that if $f\in W^{1,1}(\R)$, i.e.\ \(\cantor f=\jump f=0\), then also $\nabla\M f\in L^1(\R)$ i.e.\ \(\cantor{\M f}=\jump{\M f}=0\).
This leaves open the questions of whether the maximal function can have a nonzero Cantor part $\cantor{\M f}$.
In addition, one can ask whether a nonzero jump part $\jump{\M f}$ can be caused by anything else than a nonzero $\jump f$.
We address these two questions in the present manuscript.

In \cite{GoKo}, the authors consider the Cantor--Vitali function $f\in \BV((0,1))$,
which is the typical example of a function that does not have jumps but is not absolutely continuous either, in particular $Df=\cantor f\neq0$.
More precisely, $f$ is an increasing continuous function whose variation measure is
\[
Df=\h^{s}(C)^{-1}\restrictzero{\h^{s}}{C},
\]
where $s=\log 2/\log 3$ and $C$ is the usual $1/3$-Cantor set.
The authors show that $\M f(x)>f(x)$ for $|Df|$-a.e.\ $x\in (0,1)$, and from this it follows that
$\M f$ is absolutely continuous.
More generally, in \cite{GoKo} the authors show absolute continuity of $\M f$ for any increasing continuous function
$f\in\BV_{\loc}(\R)$
whose derivative is a finite sum of positive Radon measures $\mu_i$ that are Ahlfors regular with dimension $0<d_i<1$,
i.e.\ 
\begin{align}
\label{eq:gokoassumption}
Df
&=
\cantor f
=
\sum_{j=1}^N \mu_i
,&
\mu_i
&\geq
0
,&
\mu_i(B(x,r))
&\sim
r^{d_i}
,&
0<d_i
&<
1
.
\end{align}

In this manuscript we generalize this result substantially, see \zcref{corollary:gk24}.
We consider all dimensions $d\ge 1$,
and instead of Ahlfors regularity we consider the weaker pointwise condition that
$f$ has a 
non-concave blow-up $|\cantor f|$-almost everywhere, see \zcref{eq:blowup,defi:nonconcaveblowup}.
Then we show the following \zcref*[noref,nocap]{thm:main}.

\begin{theorem}\label{thm:main}
	Let $f\in \BV(\R^d)$ and let $R>0$.
	Let $A\subset \R^d$ be the set of points where $f$ has a non-concave blow-up.
	Then
	\[
	|\cantor f|(A\cap \{\Ml R f=f^*\})=0.
	\]
\end{theorem}

Here $\Ml R f$ denotes the maximal function where the radii in the supremum are limited to $0<r< R$ and $f^*$ the precise representative of $f$, see \zcref{eq:preciserepresentative}.
We always have $\M f\ge \Ml R f\ge f^*$.
The expression \(\{\Ml R f=f^*\}\) is a shorthand for the set of all $x\in\R^d$ with $\Ml Rf(x)=f^*(x)$.
We prove \zcref{thm:main} in \zcref{sec:main proof}.

Then in \zcref{sec:sobolev} we consider Sobolev
$W^{1,1}_{\loc}$-regularity of the maximal function.
We let $\Om$ be an open subset of $\R^d$, and then we understand $\M f$ to be the maximal function where the radii $r$ are limited
by the requirement $B(x,r)\subset \Om$.

\begin{theorem}\label{thm:main two}
	Suppose that \(f\in L^1_\loc(\Omega)\) with \(\var_\Omega f<\infty\)
	and suppose that $\M  f\in \BV_{\loc}(\Om)$.
	Then $|\jump{\M f}|\le \tfrac 12 |\jump f|$. 
	If in addition $f$ has a non-concave blow-up at $|\cantor f|$-a.e.\ point
	then $\cantor{\M f}=0$.
	If in addition $|\jump f|(\Om)=0$, then $\M f$ is
	ACL (absolutely continuous on lines) and in the class $W_{\loc}^{1,1}(\Om)$. 
\end{theorem}

Here, we have to \emph{assume} $\BV_{\loc}$-regularity of the maximal function,
because it is not known except in one dimension due to \cite[Theorem 1.2]{Kur}.

By ACL we mean absolutely continuous on almost every line parallel to a coordinate
axis;
note, that $\M f$ being in $W_{\loc}^{1,1}(\Om)$ implies the ACL property for
\emph{some} pointwise representative of $\M f$, but by \zcref{thm:main two} in fact $\M f$ itself already has this property.

\begin{corollary}\label{cor:BV to Sobo}
	Let \(f\in L^1_\loc(\R)\) with \(\var_\R f<\infty\).
	Then $|\jump{\M f}|\le \tfrac 12 |\jump f|$. 
	If in addition $f$ has a non-concave blow-up at $|\cantor f|$-a.e.\ point
	then $\cantor{\M f}=0$.
	If in addition $|\jump f|(\Om)=0$, then $\M f\in W_{\loc}^{1,1}(\R)$. 
\end{corollary}

\begin{remark}
The bound $|\jump{\M f}|\le \tfrac 12 |\jump f|$ is optimal:
For example for $f=\ind{B(0,1)}$ we have \(\jump{\M f}=\tfrac 12\jump f\).
\end{remark}

In particular, \zcref{thm:main two} generalizes the main result from \cite{GoKo} to more general functions and to all dimensions:

\begin{corollary}
\label{corollary:gk24}
	Suppose  \(f\in L^1_\loc(\Omega)\) with \(\var_\Omega f<\infty\) and $|\jump f|(\Om)=0$ such that
	\begin{enumerate}
	\item
	\label{it:singular}
	at $|\cantor f|$-a.e.\ point, $f$ has a blow-up with a singular part, or
	\item
	\label{it:ahlforsregular}
	$f:\R\rightarrow\R$ and $|\cantor f|$ is of the form of \zcref{eq:gokoassumption}.
	\end{enumerate}
	Moreover, suppose that $\M  f\in \BV_{\loc}(\Om)$.
	Then $\M f$ is ACL and in the class $W_{\loc}^{1,1}(\Om)$. 
\end{corollary}

\zcref[S]{corollary:gk24} follows from the following \zcref*[noref,nocap]{lem:singular and noncave}.

\begin{lemma}\label{lem:singular and noncave}
	Suppose $\gamma:(-1/2,1/2)\rightarrow\R$ is bounded and increasing such that
	$\singular \gamma$ (the singular part of its derivative) is nonzero. Then $\gamma$ is not concave. 
\end{lemma}
\begin{proof}
	For a contradiction suppose that $\gamma$ is concave.
	Since $\singular \gamma$ is nonzero, there exists a point $x\in (-1/2,1/2)$ such that
	by \zcref{eq:fundamental theorem of calculus for BV},
	\[
		\lim_{r\to 0}\frac{\gamma(x+r)-\gamma(x-r)}{r}
		=\lim_{r\to 0}\frac{D\gamma((x-r,x+r))}{r}\\
		=\infty.
	\]
	Consider $s\in (-1/2,1/2)$ with $s<x$, and $r>0$ small enough $s<x-r$.
	Then by concavity we have
	\[
		\gamma(s)
		\le \gamma(x-r)+(x-r-s)\frac{\gamma(x+r)-\gamma(x-r)}{2r}
		\to-\infty
	\]
	as $r\searrow 0$. This contradicts the fact that $\gamma$ is bounded.
\end{proof}

\begin{proof}[Proof of \zcref{corollary:gk24}]
Assume \zcref{it:singular}.
Then \zcref{lem:singular and noncave} implies the existence of a non-concave blow-up, and thus the conclusion follows from \zcref{thm:main two}.

Assume \zcref{it:ahlforsregular}.
We can assume all the dimensions $d_1,\ldots,d_N$ to be distinct.
Then for $|\singular f|$-a.e.\ $x\in \R$, consider the smallest $d_i$ such that $x$ is in the support of $\mu_i$.
Then we have
\[
\lim_{r\to 0}\frac{Df(B(x,r))}{\mu_i(B(x,r))}=1.
\] 
Then by a standard BV result, see \zcref{eq:blowup},
there exists a blow-up of $f$ at $x$, denoted by $w\in \BV((-1/2,1/2))$, 
such that $Dw$ is necessarily also Ahlfors $d_i$-regular.
In particular, there exists a blow-up with a nonzero singular part, i.e.\ \zcref{it:singular} holds.
\end{proof}

\begin{remark}
Note that it is not necessarily the case that $f$ has a non-concave blow-up at $|\cantor f|$-a.e.\ point.
In \cite[Example~5.9(1)]{MR890162}, Preiss gives an example of a measure on $\R$ (which can always be taken to be $Df$ for a BV function $f$)
that is singular with respect to $\lmio$, but at $|Df|$-a.e.\ point, all blow-ups of $f$
are linear functions.

It remains open whether the conclusions of \zcref{thm:main,thm:main two} continue to hold for $A=\mathbb{R}^d$ and without any assumptions on $\cantor f$ respectively.
\end{remark}

\paragraph{Acknowledgments.}
The authors wish to thank Khanh Nguyen for pointing out some useful references.

\section{Preliminaries}
\label{sec:preliminaries}

In this \zcref*[noref,nocap]{sec:preliminaries} we will follow the monograph \cite{AFP}.
The results discussed here are standard, and we only give references for some key results.

\subsection{Radon measures}\label{subsec:Radon measures}

The Borel $\sigma$-algebra on a set $H\subset \R^d$ is denoted by $\borel H$.
We always work with the Euclidean norm $|\cdot|$ for vectors $v\in \R^d$ as well as for
matrices $A\in \R^{k\times d}$, $k\ge 1$.

Let $x\in\R^d$ and $0<r<\infty$.
We denote by $B(x,r)$ the open ball in $\R^d$ with center $x$ and radius $r$.
We denote the open ball in $\R^{d-1}$ with center $z\in \R^{d-1}$ and radius $r$ by $B_{d-1}(z,r)$.
We denote a cylinder centered in $x$, parallel to $e_d$ and with length and radius $r$ by
\begin{equation}
\label{eq:cylinder}
C(x,r)
\coloneqq
B_{d-1}((x_1,\ldots,x_{d-1}),r)\times (x_d-r/2,x_d+r/2)
.
\end{equation}
Given a unit vector $v\in\Sph^{d-1}$, we denote by $C_v(x,r)$ the corresponding cylinder oriented in the direction $\nu$.

Let $\Om\subset\R^d$ be an open set and $\ell \in\N$.
We  denote by $\meas{\Om;\R^{\ell}}$ the Banach space of $\ell$-vector-valued
Radon measures. For $\mu \in \meas{\Om;\R^{\ell}}$, the
total variation $|\mu|(\Om)$ is defined
with respect to the Euclidean norm on $\R^{\ell}$.
We further denote the set of positive measures by $\measp{\Om}$.

For a vector-valued Radon measure $\nu\in\meas{\Om;\R^{\ell}}$
and a positive Radon measure $\mu\in\measp{\Om}$, we
have the Lebesgue--Radon--Nikodym decomposition
\[
\nu=\gac \nu+\gsingular \nu=\frac{d\nu}{d\mu}\intd \mu+\gsingular \nu
\]
of $\nu$ with respect to $\mu$,
where $\frac{d\nu}{d\mu}\in L^1(\Om,\mu;\R^{\ell})$.

For open sets $E\subset \R^{d-1}$, $F\subset \R$, 
a \emph{parametrized measure} $(\nu_{z})_{y \in E}$ is a map from $E$ to the set
$\meas{F;\R^{\ell}}$ of vector-valued Radon measures on $F$. For $\mu\in\measp{E}$
it is said to be
\emph{weakly* $\mu$-measurable} if $z\mapsto \nu_{z}(B)$ is
$\mu$-measurable for all Borel sets $B\in\borel{F}$. 
Suppose, that we additionally have
\(
\int_{E}|\nu_{z}|(F)\intd \mu(z)<\infty.
\)
For a set
$A \subset E\times F$
we denote its fibers by
\[
A_z\coloneqq \{t\in F \colon (z,t) \in A\},\qquad z\in E
.
\]
Then we define the generalized product measure $\mu\otimes(z\mapsto\nu_z)$ by
\begin{equation}\label{eq:product measure}
	\mu\otimes(z\mapsto\nu_z)(A)\coloneqq \int_E \nu_z(A_z)\intd \mu(z)
,
\end{equation}
for all $A\subset E\times F$ that belong to the $\sigma$-algebra on $E\times F$ generated by the set of rectangles $\borelm\mu{E} \times \borel{F}$,
 where $\borelm\mu{E}$
denotes the $\mu$-completion of $\borel{E}$.

The Lebesgue measure in $\R^d$ is denoted by $\lmo$, and the Lebesgue measure in $\R^{d-1}$ is denoted
by $\Ll^{d-1}$.
The Hausdorff measure of dimension $0\le s\le d$ is denoted by $\h^s$. 
When we integrate with respect to Lebesgue measure we may omit it in the notation and write
\[
\int f
\coloneqq
\int f\intd\lmo
=
\int f(y)\intd\lm y
=
\int f(y)\intd y
.
\]

The Sobolev $1$-capacity of a set $A\subset \R^d$ is defined by
\[
\capa_1(A)\coloneqq\inf\Bigl\{
 \int_{\R^n}|u|+|\nabla u|\intd\lmo:
u\in W^{1,1}(\R^d),\ 
u\ge 1\text{ in a nbhd of }A 
\Bigr\}
.
\]
By e.g.\ \cite[Theorem 4.3, Theorem 5.1]{HaKi} we know that for any $A\subset \R^d$,
\begin{equation}\label{eq:null sets of capa and Hausdorff}
	\capa_1(A)=0\quad\textrm{if and only if}\quad\sm A=0.
\end{equation}

We have a function and a measure theoretic notion of restrictions.
Given a map $f:X\rightarrow Y$ and $A\subset X$ we define define the function theoretic restriction $\restrict fA:A\rightarrow Y$ by $\restrict fA(x)=f(x)$.
For a measure $\mu$ on $X$ we define the measure theoretic restriction $\restrictzero\mu A$ by \(\restrictzero\mu A(B)=\mu(A\cap B)\).

Note, that while the measure theoretic restriction is in some sense a zero extension to the original domain,
our function theoretic restriction is not.
This is important for example in the following \zcref*[noref,nocap]{defi:quasicontinuous}.

\begin{definition}
\label{defi:quasicontinuous}
Let $\Om\subset \R^d$ be open.
A function $v:\Om\rightarrow[-\infty,\infty]$ is \emph{$1$-quasicontinuous} if for every $\eps>0$
there exists an open set $G\subset \Om$ such that $\capa_1(G)<\eps$
and $\restrictl v{\Om\setminus G}$ is finite and continuous.
\end{definition}

\subsection{Functions of bounded variation}

We denote by $\BV(\Omega)$ the set of all functions $f\in L_\loc^1(\Omega)$ of bounded variation,
that is, whose distributional derivative
is an $\R^{d}$-valued Radon measure with finite total variation. This means that
there exists a (necessarily unique) Radon measure $Df\in \meas{\Om;\R^{d}}$
such that for all $\varphi\in C_{\mathrm c}^1(\Omega)$, the integration-by-parts formula
\[
\int_{\Omega}f\frac{\partial\varphi}{\partial y_i}\intd\lmo
=-\int_{\Omega}\varphi\intd (Df)_i,\qquad i=1,\ldots,d
\]
holds.
Similarly to \(L^1_\loc\) we denote by \(\BV_\loc(\Om)\) the set of all $f\in L^1_\loc(\Om)$ such that for each open $G$ which is compactly contained in $\Om$ we have $\restrict fG\in\BV(G)$.

If we do not know a priori that a function $f\in L^1_{\loc}(\Om)$
is a BV function, we consider
\begin{equation}\label{eq:definition of total variation}
	\var_\Omega f \coloneqq\sup\left\{\int_{\Om}f\cdot \nabla \varphi\intd\lmo,\,\varphi\in C_c^{1}(\Om),
	\,|\varphi|\le 1\right\}.
\end{equation}
If $\var_\Omega f <\infty$, then the Radon measure $Df$ exists and $\var_\Omega f=|Df|(\Om)$
by the Riesz representation theorem.

For a measurable $A\subset\R^d$ with $0<\lm A<\infty$ and $f\in L^1(A)$,  we denote the integral average by
\begin{equation}
\label{eq:integralaverage}
f_A
=
\vint_A f
=
\frac1{\lm A}
\int_A f
.
\end{equation}
The precise representative of $f$ is defined by
\begin{equation}
\label{eq:preciserepresentative}
f^*(x)\coloneqq\limsup_{r\to 0}f_{B(x,r)},\quad x\in \Om.
\end{equation}
This is easily seen to be a Borel function.
We say that $x\in\Om$ is a Lebesgue point of $f$ if
\[
\lim_{r\to 0}\,\vint_{B(x,r)}|f(y)-f^*(x)|\intd y=0
.
\]
We denote by $S_f\subset\Om$ the set where
this condition fails and call it the approximate discontinuity set.

We say that $x\in \Om$ is an approximate jump point of $f$ if there exists a $\nu\in\Sph^{d-1}$
and distinct numbers $f^+(x), f^-(x)\in\R$ such that
\begin{equation}
\label{eq_approximatejumpdefinition}
	\lim_{r\to 0}\,\vint_{\{y\in B(x,r)\colon \langle y-x,\pm \nu\rangle>0\}}|f(y)-f^\pm(x)|\intd y=0
.
\end{equation}
The set of all approximate jump points is denoted by $J_f$.

The lower and upper approximate limits of a function $f\in\BV_{\loc}(\Om)$
are defined respectively by
\[
f^{\wedge}(x)\coloneqq 
\sup\left\{t\in\R\colon \lim_{r\to 0}\frac{\lm{B(x,r)}\cap\{f<t\})}{\lm{B(x,r)}}=0\right\}
\]
and
\[
f^{\vee}(x)\coloneqq 
\inf\left\{t\in\R\colon \lim_{r\to 0}\frac{\lm{B(x,r)\cap\{f>t\}}}{\lm{B(x,r)}}=0\right\},
\]
for all $x\in\Om$.
We interpret the supremum and infimum of an empty set to be $-\infty$ and $\infty$, respectively.
Note, that
\begin{equation}\label{eq:representatives outside jump set}
f^{*}(x)=f^{\wedge}(x)=f^{\vee}(x)
\qquad\textrm{for }x\in \Om\setminus S_f,
\end{equation}
and
\begin{equation}\label{eq:representatives in jump set}
f^{\wedge}(x)=\min\{f^{-}(x),f^+(x)\}
,
\qquad
f^{\vee}(x)=\max\{f^{-}(x),f^+(x)\}
\qquad\textrm{for }
x\in J_f.
\end{equation}

We write the Radon-Nikodym decomposition of the variation measure of $f$ into
the absolutely continuous and singular parts with respect to Lebesgue measure
$\lmo$ as $Df=\ac f+\singular f$. Furthermore, we define the Cantor and jump parts of $Df$ by
\begin{align}\label{eq:Dju and Dcu}
	\cantor f
	&\coloneqq
	\restrictzerol{\singular f}{\Om\setminus S_f}
	,&
	\jump f
	&\coloneqq
	\restrictzero{\singular f}{J_f}
	.
\end{align}
We have
\begin{equation}\label{eq:Sf and Jf}
\sm{S_f\setminus J_f}=0,
\end{equation}
and $|Df|$ vanishes on
$\smo$-negligible sets, and so we have the decomposition (see \cite[Section 3.9]{AFP})
\begin{equation}\label{eq:decomposition of variation measure}
	Df=\ac f+ \cantor f+ \jump f.
\end{equation}
From the fact that  $\sm{S_f\setminus J_f}=0$, we also get
\begin{equation}\label{eq:f star and vee}
f^*(x)\le f^{\vee}(x)\quad\textrm{for }\smo\textrm{-a.e.\ }x\in \Om.
\end{equation}
For the jump part, we know
\begin{equation}\label{eq:jump part representation}
	|\jump f|=|f^{+}-f^-|\restrictzero\smo{J_f}.
\end{equation}

\subsection{One-dimensional sections of $\BV$ functions}\label{subsec:one dimensional sections}

The following notation and results on one-dimensional sections of
$\BV$ functions are given in \cite[Section 3.11]{AFP}.

Let $d=1$. Suppose $f\in\BV_{\loc}(\R)$.
We have $J_f=S_f$, $J_f$ is at most countable,
and $Df (\{x\})=0$ for every $x\in\R\setminus J_f$.
For every $x,y\in \R\setminus J_f$ with $x<y$ we have
\begin{equation}\label{eq:fundamental theorem of calculus for BV}
	f^*(y)-f^*(x)=Df((x,y)).
\end{equation}
For every $x,y\in \R$ with $x<y$ we have
\begin{equation}\label{eq:fundamental theorem of calculus for BV 2}
	|f^*(x)-f^*(y)|\le |Df|([x,y]),
\end{equation}
and the same with $f^*$ replaced by $f^{\wedge}$ or $f^{\vee}$, or any pairing of these.
We say that
\[
f^{\vee}(x)-f^{\wedge}(x)=|Df|(\{x\})
\]
is the
\emph{jump size} of $f$ at point $x$.

We denote by $\pi:\R^d\rightarrow\R^{d-1}$ the projection given by
\[
\pi(x)\coloneqq(x_1,\ldots,,x_{d-1})
.
\]
We denote by $e_1,\ldots,e_d$ the standard basis vectors in $\R^d$.
For a set $A \subset \R^d$ we denote its fibers by
\[
A_z\coloneqq \{t\in\R \colon (z,t) \in A\},\qquad z\in \pi(A).
\]
For a map $f$ defined on $\Om$ we denote
\[
f_z(t)\coloneqq f(z,t),\quad t\in \Om_z,\ z\in \pi(\Om).
\]
For $f\in\BV(\Om)$ we we have $f_z\in \BV(\Om_z)$ for $\lmlo$-almost every $z\in\pi(\Om)$, see \cite[Theorem 3.103]{AFP}. 

Recall \zcref{eq:product measure}.
Denoting $D_d f\coloneqq \langle Df,e_d\rangle$, we further have
\begin{align*}
D_d f&=\lmlo\otimes(z\mapsto D(f_z))
,&
\sjump df&=\lmlo\otimes(z\mapsto \jump{(f_z)}),
\end{align*}
see \cite[Theorem 3.107 \& Theorem 3.108]{AFP}.
It follows that
\begin{align}\label{eq:slice representation for total variation}
	|D_d f|&=\lmlo\otimes(z\mapsto |D(f_z)|)
	,&
	|\sjump df|&=\lmlo\otimes(z\mapsto |\jump{(f_z)}|),
\end{align}
see \cite[Corollary 2.29]{AFP}, and similarly
\begin{equation}\label{eq:sections and jump sets 3}
|D^c_d f|=\lmlo\otimes(z\mapsto |\cantor{(f_z)}|).
\end{equation}
Moreover, for $\lmlo$-almost every 
$z\in\pi(\Om)$ we have
\begin{align}\label{eq:sections and jump sets}
	J_{f_{z}}&=(J_f)_{z}
	,&
	(f^*)_{z}(t)&=(f_{z})^*(t)\ \ \textrm{for every }t\in \R\setminus J_{f_{z}},
\end{align}
see \cite[Theorem 3.108]{AFP}.
By \cite[Theorem 3.108]{AFP} we also know that for $\lmlo$-almost every 
$z\in\pi(\Om)$, we have
\begin{equation}\label{eq:sections and jump sets 2}
\{(f_z)^{-}(t),(f_z)^{+}(t)\}=\{(f^{-})_z(t),(f^{+})_z(t)\}\quad\textrm{for every }t\in (J_f)_z.
\end{equation}
Combining \zcref{eq:sections and jump sets,eq:sections and jump sets 2} with
\zcref{eq:representatives outside jump set,eq:representatives in jump set},
for $\lmlo$-almost every $z\in\pi(\Om)$ and every $t\in \Om_z$ we have
\begin{align}\label{eq:upper and lower repr sections}
	(f_z)^{\wedge}(t)&=(f^{\wedge})_z(t)
	,&
	(f_z)^{\vee}(t)&=(f^{\vee})_z(t)
	,&
	(f_z)^{*}(t)&=(f^{*})_z(t)
	.
\end{align}

\subsection{Blow-ups and nonconcavity}

Let $\Om\subset \R^d$ be an open set and suppose $f\in L^1_{\loc}(\Om)$. 
Recall the definition \zcref{eq:integralaverage} of the integral average $f_{B(x,r)}$ and for $x\in\R^d$ and $A\subset\R^d$ define
\[
\dist(x,A)
\coloneqq
\inf_{y\in A}
|x-y|
.
\]
We define the local maximal function $\M_\Om f$ by
\begin{equation}\label{eq:HL def}
	\M_\Om f(x)\coloneqq
	\sup_{0<r<\dist(x,\R^d\setminus\Om)}
	f_{B(x,r)}
	\qquad x\in \Om.
\end{equation}
We usually drop $\Om$ from the notation, the exception being
when instead of the (fixed) open set $\Om$ we consider the maximal function in a different set, e.g.\ a cylinder $C(x,r)$.
Note that in most of the literature one takes an absolute value of $f$ in the integral;
with our definition, this corresponds essentially to the special case of nonnegative functions.
For $R>0$ we define the auxiliary maximal functions
\begin{align}
\notag
\Ml R f(x)
&\coloneqq
\sup_{0<r<\min\{R,\dist(x,\R^d\setminus\Om)\}}
f_{B(x,r)}
,\\
\label{eq:auxiliary maximal operator}
\Mg R f(x)
&\coloneqq
\sup_{R\leq r<\dist(x,\R^d\setminus\Om)}
f_{B(x,r)}
,
\end{align}
so that $\M f(x)=\max\{\Ml Rf(x),\Mg Rf(x)\}$.
Note, that $\Mg Rf(x)$ is only defined for those $x\in\Om$ with $B(x,R)\subset\Om$.

\begin{definition}\label{def:superharmonic}
	Let $\Om\subset \R^d$ be open and let $f\in L^1_{\loc}(\Om)$.
	We say that $f$ is $2$-superharmonic if it is lower semicontinuous
	and
	\[
	f(x)\ge \M_{\Om}f(x)\qquad \textrm{for all }x\in \Om.
	\]
\end{definition}

\begin{definition}
	Let $\Om\subset \R^d$ be open and let $f\in L^{1}_{\loc}(\Om)$.
	We say that $f$ is a $2$-supersolution if
	\[
	\int_{\Om} f \Delta \varphi\le 0
	\qquad
	\text{for all }\varphi\in C_c^{2}(\Om).
	\]
\end{definition}

For the following \zcref*[noref,nocap]{thm:superharmonic and supersol} see e.g.\ Theorems 2.59 and 2.65 in the monograph \cite{MR1461542}.

\begin{theorem}\label{thm:superharmonic and supersol}
	Let $\Om\subset \R^d$ be open and let $f\in L^{1}_{\loc}(\Om)$.
	If $f$ is $2$-superharmonic, then it is a $2$-supersolution.
	Conversely, if $f$ is a $2$-supersolution, then there exists a function $\widetilde{f}=f$ a.e.\
	which is  $2$-superharmonic.
\end{theorem}

We have the following quasi-semicontinuity result from \cite[Theorem 2.5]{CDLP}.
Alternatively, see
\cite[Theorem 1.1]{LaSh} and \cite[Corollary 4.2]{L-SA}
for a proof of \zcref{thm:quasisemicontinuity} in more general metric spaces.

\begin{theorem}\label{thm:quasisemicontinuity}
	Let $\Om\subset \R^d$ be open, let $f\in\BV_{\loc}(\Om)$ and let $\eps>0$.
	Then there exists an open set $G\subset\Om$
	such that $\capa_1(G)<\eps$ and 
	the map $\restrict{f^{\wedge}}{\Om\setminus G}$
	is finite and lower semicontinuous,
	and $\restrict{f^{\vee}}{\Om\setminus G}$ is finite and upper
	semicontinuous.
\end{theorem}

\begin{lemma}\label{lem:capacity and Hausdorff measure}
	Let $A\subset \R^d$. Then
	\[
	2\sm{\pi(A)}\le \capa_1(A).
	\]
\end{lemma}
\begin{proof}
	Consider $u\in W^{1,1}(\R^d)$ with $u\ge 1$ in a neighborhood of $A$.
	Then for $\lmlo$-a.e.\ $z\in \pi(A)$, we have
	\[
	\int_{A_z}\left|\frac{\partial u}{\partial x_d}\right|\ge 2.
	\]
	Integrating over all $z\in\R^{d-1}$, we get
	\[
	\int_{\R^d}|\nabla u|\ge 2\sm{\pi(A)}.
	\]
	Thus $\Vert u\Vert_{W^{1,1}(\R^d)}\ge 2\sm{\pi(A)}$ and
	we get the result by taking infimum over all such $u$.
\end{proof}

Recall that for a BV function $f$ on the real line,
we call $f^{\vee}(x)-f^{\wedge}(x)=|Df|(\{x\})$ the jump size of $f$ at point $x\in\R$.

\begin{lemma}\label{lem:biggest jump}
	Let $\Om\subset \R^d$ be open and let $f\in \BV(\Om)$. 
	For every $z\in \pi(\Om)$ such that
	$f_z\in \BV(\Om_z)$ (which holds for $\lmlo$-a.e.\ $z\in  \pi(\Om)$),
	let $\alpha(z)$ be the maximal size of all jumps of $f_z$.
	Then the map $\alpha:\pi(\Om)\rightarrow[0,\infty]$ is $\lmlo$-measurable.
\end{lemma}
\begin{proof}
	Note that $|D(f_z)|$ is weakly* $\lmlo$-measurable;
	recall \zcref{subsec:one dimensional sections}.
	For a.e.\ $z\in \pi(\Om)$
	we have
	\[
	\alpha(z)= \lim_{k\to\infty}\max_{i\in\Z} |D(f_z)|([i2^{-k},(i+1)2^{-k}]\cap \Om_z)
	.
	\]
\end{proof}

Consider $f\in \BV_{\loc}(\R^d)$.
For $|D f|$-a.e.\ and in particular for 
$|\cantor f|$-a.e.\ $x\in \R^d$ the Radon--Nikodym derivative exists:
\begin{equation}
\label{eq:radon-nikodym}
\frac{\intd Df}{\intd |Df|} (x)
=
\lim_{r\to 0}\frac{Df(B(x,r))}{|Df|(B(x,r))}
\eqqcolon
\xi(x)
\end{equation}
with $|\xi(x)|=1$.
 Recall the integral average $f_{C_{\xi(x)}(x,r)}$ over the cylinder $C_{\xi(x)}(x,r)$, see
\zcref{eq:cylinder,eq:integralaverage}.
For $r>0$ define the rescaling
\begin{equation}\label{eq:scalings def}
f_{r}(y)\coloneqq\frac{f(x+ry)-f_{C_{\xi(x)}(x,r)}}{|Df|(C(x,r))/r^{d-1}}
,
\qquad y\in C_{\xi(x)}(0,1)
.
\end{equation}
The following blow-up behavior is known:

\begin{theorem}[{\cite[Theorem 3.95]{AFP}}]
\label{eq:blowup}
Let $f:\R^d\rightarrow\R$ for which $|Df|$ is a Radon measure.
Then for $|Df|$-a.e.\ $x\in\R^d$
there exists a sequence $r_1,r_2,\ldots>0$ with $r_i\to 0$ and an increasing, nonconstant $\gamma:(-1/2,1/2)\rightarrow\R$ such that for $w:C_{\xi(x)}(0,1)\rightarrow\R$ given by
\(
w(y)
=
\gamma(\langle y, \xi(x)\rangle)
\)
we have
\[
f_{r_i}
\to
w
\qquad\text{in }L^1(C_{\xi(x)}(0,1))
.
\]
\end{theorem}

\begin{definition}
\label{defi:nonconcaveblowup}
We call $\gamma$ or $w$ from \zcref{eq:blowup} a \emph{blow-up} of $f$ at $x$.

There may exist several different blow-ups $\gamma$ of $f$ in $x$, and if one of them is not concave
then we say that $f$ has a non-concave blow-up at $x$.
\end{definition}

For $x\in\mathbb{R}^d$ and $r>0$ denote by $\C(x,r)$ the set of cylinders with length and radius $r$ that are centered in $x$.
For Radon measures $\mu,\nu$ denote
\begin{align*}
\overline{D(\C)}_\mu\nu(x)
&=
\limsup_{r\rightarrow0}
\sup_{C\in\C(x,r)}
\frac{
\nu(C)
}{
\mu(C)
}
,\\
\underline{D(\C)}_\mu\nu(x)
&=
\liminf_{r\rightarrow0}
\inf_{C\in\C(x,r)}
\frac{
\nu(C)
}{
\mu(C)
}
.
\end{align*}

Since (closed) cylinders also satisfy the covering theorem \cite[Theorem~1.28]{MR3409135},
due to Morse's Covering Theorem (see e.g. \cite[Theorem 1.147]{FoLe}),
we can prove the following \zcref*[noref,nocap]{lemma:measurederivative} the same way as \cite[Lemma~1.2]{MR3409135}.
\begin{lemma}
\label{lemma:measurederivative}
Let $0<\alpha<\infty$.
Then
\begin{enumerate}
\item
\(A\subset\{\underline{D(\C)}_\mu\nu\leq\alpha\}\)
implies
\(\nu(A)\leq\alpha\mu(A)\)
,
\item
\(A\subset\{\overline{D(\C)}_\mu\nu\geq\alpha\}\)
implies
\(\nu(A)\geq\alpha\mu(A)\)
.
\end{enumerate}
\end{lemma}

\begin{lemma}
\label{lemma:densitysingular}
Let $\mu,\nu$ be Radon measures that are mutually singular.
Then
\begin{align*}
\text{For $\nu$-a.e. }&x
&
\underline{D(\C)}_{\mu+\nu}\nu(x)
=
\overline{D(\C)}_{\mu+\nu}\nu(x)
&=
1
,\\
\text{For $\mu$-a.e. }&x
&
\underline{D(\C)}_{\mu+\nu}\nu(x)
=
\overline{D(\C)}_{\mu+\nu}\nu(x)
&=
0
.
\end{align*}
\end{lemma}

\begin{proof}
Since $\mu$ and $\nu$ are mutually singular there exists a set $B$ with \(\nu=\restrictzero\nu B\) and \(\mu(B)=0\).
As a consequence of \zcref{lemma:measurederivative}, the set of all points $x\in B$ with \(\underline{D(\C)}_\mu\nu(x)<\infty\) must have $\nu(B)=0$.
This means for $\nu$-almost every $x$ we have \(\underline{D(\C)}_\mu\nu(x)=\infty\) and thus
\[
1
\geq
\overline{D(\C)}_{\mu+\nu}\nu(x)
\geq
\underline{D(\C)}_{\mu+\nu}\nu(x)
=
1
.
\]
The second statement follows similarly.
\end{proof}

For $|Df|$-almost every $x\in\Omega$ the Radon--Nikodym derivative \zcref{eq:radon-nikodym} exists.
Since $\cantor f$ and $Df-\cantor f$ are mutually singular, it follows from \zcref{lemma:densitysingular} that
for $|\cantor f|$-almost every $x\in\Omega$ we have
\begin{equation}\label{eq:Cantor density}
\lim_{r\to 0}\frac{|\cantor f|(C_{\xi(x)}(x,r))}{|Df|(C_{\xi(x)}(x,r))}=1
.
\end{equation}

\section{Proof of \texorpdfstring{\zcref{thm:main}}{Theorem~1.1}}\label{sec:main proof}

Fix a point $x\in A$.
Discarding a set of $|Df|$-measure zero, 
we can assume that the Radon--Nikodym derivative \zcref{eq:radon-nikodym} exists with $|\xi(x)|=1$ and that the conclusion of \zcref{eq:blowup} and \zcref{eq:Cantor density} hold.
Rotate, so that $\xi(x)=e_d$.

Recall the definition \zcref{eq:HL def} and
note that $\M_{C(0,1)}w(y) \ge w^{*}(y)$ for all $y\in C(0,1)$.
The main ingredient is the following simple \zcref*[noref,nocap]{lem:mwlwsomewhere}.

\begin{lemma}
\label{lem:mwlwsomewhere}
$\M_{C(0,1)}w(y) > w^{*}(y)$ for some $y\in C(0,1)$.
\end{lemma}

\begin{proof}
For a contradiction suppose that
$\M_{C(0,1)}w(y) \le w^{*}(y)$ for all $y\in C(0,1)$, which means so that in fact equality holds.
Then $w^{*}$ is $2$-superharmonic on $C(0,1)$,
which means $w$ is a $2$-supersolution by \zcref{thm:superharmonic and supersol}. 
But since $w$ only depends on the $d$th coordinate,
necessarily $\gamma$ is  a $2$-supersolution on $(-1/2,1/2)$.
Thus by \zcref{thm:superharmonic and supersol},
there is a function $\widetilde{\gamma}=\gamma$ a.e.\ in $(-1/2,1/2)$, which is $2$-superharmonic 
and thus necessarily a concave function
(see e.g. \cite[Theorem 3.2.2]{MR1801253})
Since $\gamma$ is increasing, necessarily $\widetilde{\gamma}=\gamma$ everywhere in $(-1/2,1/2)$.
But this contradicts the fact that $\gamma$ is not concave, finishing the proof.
\end{proof}

By \zcref{lem:mwlwsomewhere} there exists a point
\(y'=(z',t')\in B_{d-1}(0,1)\times(-1/2,1/2)\),
for which $\M_{C(0,1)}w(y') > w^{*}(y')$.
Thus there exists $0<r\le 1/2$ such that $B(y',r)\subset C(0,1)$ and
\[
w_{B(y',r)}> w^*(y').
\]
Recall that $w$ is bounded.
By continuity, that means there is an $0<r'< 1/2$ such that
\begin{equation}\label{eq:delta def}
\delta
\coloneqq
w_{B(y',r')}- w^*(y')
>0
.
\end{equation}
Since $w$ only depends on the last coordinate, this inequality remains true for all $y=(z,t')$ with $z\in B_{d-1}(0,1-r')$.

To describe the idea of what we do next, assume for the moment that $\gamma$ was continuous so that $w^*=w$.
Then for
\[
t'_1
=
\inf\{t:\gamma(t)\geq\gamma(t')+\delta\}
>
t'
\]
we have
\(
w(z',t'_1)
=
w(y')+\delta
\)
and \(w(z',t)<w(y')+\delta\) for $t<t'_1$.
Since $w$ is increasing in the last coordinate so is $y\mapsto w_{B(y,r')}$, which means that
for all $y\in B_{d-1}(0,1-r')\times(t',t'_1)$ we would have (supposing that $B(y,r')\subset C(0,1)$)
\[
\M_{C(0,1)}w(y)
\geq
w_{B(y,r')}
\geq
w_{B(y',r')}
=
w(y')+\delta
>
w(y)
.
\]
\begin{figure}
\centering
\includegraphics{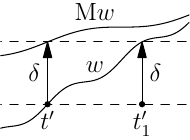}
\caption{Neighborhood of $t'$ where \(\M_{C(0,1)}w(t')-w^{*}(t')=\delta>0\).}
\label{fig:delta}
\end{figure}
Using 
\[
|D\gamma'|((t',t'_1))
=
\gamma(t'_1)-\gamma(t')
=
\delta,
\]
we could deduce that the maximal function is larger than the function itself in a $|Dw|$-large set:
\begin{align*}
|Dw|(C(0,1)\cap \{\M_{C(0,1)}w>w^{*}\})
&\geq
|Dw|(B_{d-1}(0,1-r')\times(t',t'_1))
\\
&=
\delta
\lml{B_{d-1}(0,1-r')}
,
\end{align*}
see \zcref{fig:delta}.

The
following \zcref*[noref,nocap]{lem:positivedensity} achives the corresponding estimate for the approximants $f_{r_i}$ of $w$.

\begin{lemma}
\label{lem:positivedensity}
For $\delta,r'>0$ from \zcref{eq:delta def} we have
\[
\limsup_{i\rightarrow\infty}
|Df_{r_i}|(C(0,1)\cap \{\M_{C(0,1)} f_{r_i}> f_{r_i}^*\})
>
\tfrac{\delta}{10}
\lml{B_{d-1}(0,1-r')}
.
\]
\end{lemma}

Before we prove \zcref{lem:positivedensity} we use it to conclude \zcref{thm:main}.

\begin{proof}[Proof of \zcref{thm:main}]
First, note that $\{\Ml R f>f^*\}$ is a Borel set and thus $|Df|$-measurable.
Since $r_i\rightarrow0$, in particular $r_i<R$ for $i$ large enough.
Recall the definition of the rescalings \zcref{eq:scalings def}.
Then
\begin{align*}
&\limsup_{r\to 0}\frac{|Df|(C(x,r)\cap \{\Ml R f>f^*\})}{|Df|(C(x,r))}\\
&\qquad \ge \limsup_{i\rightarrow\infty}
\frac{|Df|(C(x,r_i)\cap \{\M_{C(x,r_i)} f>f^*\})}{|Df|(C(x,r_i))}\\
&\qquad = \limsup_{i\rightarrow\infty}
|Df_{r_i}|(C(0,1)\cap \{\M_{C(0,1)} f_{r_i}> f_{r_i}^*\})
\\
&\qquad \ge \tfrac{\delta}{10}
\lml{B_{d-1}(0,1-r')}
>0.
\end{align*}
Due to \zcref{eq:Cantor density} this means
\[
\limsup_{r\to 0}\frac{|\cantor f|(C(x,r)\cap \{\Ml R f>f^*\})}{|\cantor f|(C(x,r))}>0.
\]
By \zcref{lemma:densitysingular}
this implies
\[
\lim_{r\to 0}\frac{|\cantor f|(C(x,r)\cap \{\Ml R f=f^*\})}{|\cantor f|(C(x,r))}=1
\]
for $|\cantor f|$-a.e.\ $x\in \{\Ml R f=f^*\}$.
Thus necessarily
$|\cantor f|(A\cap \{\Ml R f=f^*\})=0$.
\end{proof}

It remains to prove \zcref{lem:positivedensity}.

\begin{proof}[Proof of \zcref{lem:positivedensity}]
Since $|Df_r|(C(0,1))=1$, by lower semicontinuity we get $|Dw|(C(0,1))\le1$, and then
by \zcref{eq:slice representation for total variation},
necessarily $|D\gamma|((-1/2,1/2))\le1/\lml{B_{d-1}(0,1)}$, and so $\gamma$ and $w$ are bounded.
The convergence $f_{r_i}\to w$ in $L^1(C(0,1))$ implies pointwise $\lmo$-a.e.\ convergence for a subsequence (not relabeled).
By \zcref{eq:representatives outside jump set} this means also
$f_{r_i}^{\vee}(y)\to w^{*}(y)=w^{\vee}(y)$ as $i\to\infty$ for $\lmo$-a.e.\ $y\in C(0,1)$.

From \zcref{eq:upper and lower repr sections} recall
$((f_{r_i})^{\wedge})_z(t)=((f_{r_i})_z)^{\wedge}(t)$
for $\lmlo$-a.e.\ $z\in B_{d-1}$ and every $t\in (-1/2,1/2)$.
Thus we can simply use the notation $(f_{r_i})^{\wedge}_z(t)$, and
$(f_{r_i})^{\vee}_z(t)$.
Similarly by \zcref{eq:upper and lower repr sections},
for $\lmlo$-a.e.\ $z\in B_{d-1}$ and every $t\in (-1/2,1/2)$ we have
\[
	\gamma^*(t)=w^*(z,t)
	,
\]
which thus in fact holds for all $(z,t)\in C(0,1)$.

From \zcref{eq:delta def} and the
continuity of the integral we find a $-1/2<t<t'$, and $0<r''\leq r'-(t'-t)$, such that
for all $z\in (1-r') B_{d-1}$ we still have
\[
w_{B((z,t),r'')}> w^{*}(z,t)+4\delta/5
,
\]
and for $\lmlo$-a.e.\ $z\in B_{d-1}$ we have
$f_{r_i}^{\vee}(z,t)\to w^{\vee}(z,t)=w^*(z,t)$ as $i\to\infty$.
We also find $t+r''<s<1/2$
such that $f_{r_i}^{\vee}(z,s)\to w^{\vee}(z,s)$ as $i\to\infty$ for $\lmlo$-a.e.\
$z\in B_{d-1}$.
Then for all $v\in [t,s]$, since $\gamma$ is increasing, we have
\begin{equation}\label{eq:estimate for h}
	w_{B((z,v),\min\{r'',1/2-v\})}> w^{*}(z,t)+4\delta/5.
\end{equation}
By the convergence $f_{r_i}\to w$ in $L^1(C(0,1))$ as $i\to\infty$, we have
\[
(f_{r_i}-w)_{B((z,v),\min\{r'',1/2-v\})}
\to
0
\]
uniformly for all $z\in (1-r') B_{d-1}$ and $v\in [t,s]$.
Thus for sufficiently large $i$, we have
\begin{equation}\label{eq:uniform closeness}
	|
	(f_{r_i})_{B((z,v),\min\{r'',1/2-v\})}
	-
	w_{B((z,v),\min\{r'',1/2-v\})}
	|
	<\delta/5
\end{equation}
Consider the set $D_i$ of those $z \in B_{d-1}$ for which
\begin{equation}\label{eq:only small jumps}
	(f_{r_i})_z\in\BV((-1/2,1/2))\textrm{ has jumps at most size }
	\ \delta/5.
\end{equation}
By \zcref{lem:biggest jump}, $D_i$ is $\lmlo$-measurable.
Due to \zcref{eq:Cantor density} we have $|\jump{f_{r_i}}|(C(0,1))\rightarrow0$ and thus by \zcref{eq:slice representation for total variation} we have
\begin{equation}\label{eq:size of Dj sets}
	\lim_{i\to\infty}\lml{B_{d-1}\setminus D_i}= 0.
\end{equation}
By Egorov's theorem, we find an $\lmlo$-measurable set
$H\subset (1-r') B_{d-1}$ such that $\lml H> \frac 12 \lml{(1-r') B_{d-1}}$,
and $f_{r_i}^{\vee}(z,t)\to w^{*}(z,t)$
and $f_{r_i}^{\vee}(z,s)\to w^{\vee}(z,s)$
as $i\to\infty$ uniformly for all $z\in H$.
From \zcref{eq:estimate for h,eq:uniform closeness} we then get for sufficiently large $i$,
for all $z\in H$ and for all $v\in [t,s]$, that
\begin{equation}\label{eq:max func much bigger}
	(f_{r_i})_{B((z,v),\min\{r'',1/2-v\})}\ge f_{r_i}^{\vee}((z,t))+3\delta/5.
\end{equation}
For all sufficiently large $i$ and for all $z\in H$, we also have
\begin{align*}
	f_{r_i}^{\vee}(z,s)
	&\ge w^{\vee}(z,s)-\delta/5\\
	&\ge w_{B((z,t),r'')}-\delta/5&& \text{since $\gamma$ is increasing}\\
	&>(f_{r_i})_{B((z,t),r'')}
	-2\delta/5&&\text{by \zcref{eq:uniform closeness}}\\
	&>f_{r_i}^{\vee}((z,t))+\delta/5&&\text{by \zcref{eq:max func much bigger}}\\
	&\ge f_{r_i}^{\wedge}((z,t))+\delta/5.
\end{align*}
By \zcref{lem:capacity and Hausdorff measure}
and \zcref{thm:quasisemicontinuity} we can assume that
$(z,t)\mapsto(f_{r_i})_z^{\vee}(t)$ restricted to $H\times (-1/2,1/2)$ is upper semicontinuous,
and that $(z,t)\mapsto(f_{r_i})_z^{\wedge}(t)$
restricted to $H\times (-1/2,1/2)$
is lower semicontinuous, for every $i\in\N$.
For all $z\in H$ and for all sufficiently large $i$,
we find the smallest $t\le t_{z}\le s$ such that
\[
f_{r_i}^{\vee}((z,t_{z}))
\ge f_{r_i}^{\wedge}((z,t))+\delta/5.
\]
Note that now
\begin{equation}\label{eq:lsc}
	H\ni z\mapsto t_z\quad\textrm{is lower semicontinuous.}
\end{equation}
By \zcref{eq:only small jumps}
for every $z\in H\cap D_i$ and $v\in [t,t_{z}]$ we have
\begin{equation}\label{eq:choice of tjz}
	f_{r_i}^{\vee}((z,v))\\
	\le  f_{r_i}^{\wedge}((z,t))+2\delta/5.
\end{equation}
By \zcref{eq:fundamental theorem of calculus for BV 2}, we have
\begin{align*}
	|D[(f_{r_i})_z]|([t,t_{z}])
	&\ge (f_{r_i}^{\vee})_z(t_{z})-(f_{r_i}^{\wedge})_z(t)
	\ge \delta/5.
\end{align*}
By \zcref{eq:max func much bigger,eq:choice of tjz}, we have	
for all $v\in [t,t_{z}]$,
\[
(f_{r_i})_{B((z,v),\min\{r'',1/2-v\})}
\ge f_{r_i}^{\vee}((z,v))+\delta/5.
\]
Recalling also \zcref{eq:f star and vee}, for every $z\in H\cap D_i$ we get
\[
|D[(f_{r_i})_z]|([t,t_{z}] \cap \{\M_{C(0,1)} f_{r_i}> f_{r_i}^*\})
\ge \delta/5.
\]
Integrating over $H\cap D_i$,
where the required measurability is guaranteed by \zcref{eq:lsc},
by \zcref{eq:slice representation for total variation} we get
\[
|Df_{r_i}|(C(0,1)\cap \{\M_{C(0,1)} f_{r_i}> f_{r_i}^*\}) 
\ge  \tfrac{\delta}{5}\lml {H\cap D_i}
> \tfrac{\delta}{10}\lml {(1-r') B_{d-1}}
\]
for sufficiently large $i$, due to \zcref{eq:size of Dj sets}.
\end{proof}

\section{Sobolev regularity}\label{sec:sobolev}

In this \zcref*[noref,nocap]{sec:sobolev}, we will use the results proved in \zcref{sec:main proof} to prove Sobolev regularity
of the maximal function of certain functions of bounded variation.

\subsection{Quasicontinuity}

First, we prove that $\M f$ is 1-quasicontinuous outside of $J_f$, see \zcref{defi:quasicontinuous,thm:quasicontinuity of noncentered maximal function}.
The proof of \zcref{thm:quasicontinuity of noncentered maximal function} is analogous to the one given in \cite[Theorem 4.1]{panubvsobolev} for the uncentered maximal function;
we note that a few
statements in \cite[Sections 3-5]{panubvsobolev} were incorrect since only $f\in \BV_{\loc}(\Om)$
was assumed, when in fact $\var f<\infty$ is needed.

As before, $\Om\subset \R^d$ is always an open set.
We start with some preliminary results.

\begin{lemma}[{\cite[Theorem~107]{lecturenoteshajlasz}}]
\label{johnpoincare_original}
Let $\Om$ be a John domain and let $f\in W^{1,1}(\Omega)$.
Then $f\in L^{d/(d-1)}(\Om)$ and
\[
\Bigl(
\int_\Om
|f-f_\Om|^{\frac d{d-1}}
\Bigr)^{\frac{d-1}d}
\lesssim
\int_\Om
|\nabla f|
.
\]
\end{lemma}

We require the following variant.

\begin{lemma}
\label{johnpoincare}
Let $\Om$ be a John domain and let $f\in L^1_\loc(\Om)$ with $\var_\Om f<\infty$.
Then $f\in L^1(\Om)$ and
\[
\frac1{\lm\Om^{\frac1d}}
\int_\Om
|f-f_\Om|
\lesssim
\var_\Om f
.
\]
\end{lemma}

\begin{proof}
Observe, that for $f\in L^1(\Omega)$ we have
\begin{equation}
\label{eq:fomegavsinf}
\inf_{c\in\mathbb{R}}
\int_\Omega
|f-c|
\leq
\int_\Omega
|f-f_\Om|
\leq
\inf_{c\in\mathbb{R}}
\int_\Omega
|f-c|
+
|f_\Omega-c|
\leq
2
\inf_{c\in\mathbb{R}}
\int_\Omega
|f-c|
.
\end{equation}
As a consequence of Cavalieri's principle,
\[
\int_\Omega
|f-c|
=
\int_{-\infty}^c
\lm{\Omega\cap\{f\leq\lambda\}}
\intd\lambda
+
\int_c^\infty
\lm{\Omega\cap\{f>\lambda\}}
\intd\lambda
,
\]
the infimum is attained for the median,
\[
c
=
\median \Omega f
\coloneqq
\inf\{\lambda\in\R:\lm{\Omega\cap\{f>\lambda\}}\leq\lm \Omega/2\}
,
\]
which is finite for any $f$ that is finite almost everywhere.

Let $f$ be bounded.
Since a John domain is by definition bounded, we have $f\in L^1(\Om)$,
and by e.g.\ \cite[Theorem 3.9]{AFP} we find a sequence of functions $f_i\in W^{1,1}(\Om)$ with $f_i\to f$ in $L^1(\Om)$
and $\var_\Om f_i\to \var_\Om f$.
By \zcref{eq:fomegavsinf}, H\"older's inequality and \zcref{johnpoincare_original} we can conclude 
\[
\frac1{\lm\Om^{\frac1d}}
\int_\Om
|f-\median \Omega f|
\leq
\frac1{\lm\Om^{\frac1d}}
\int_\Om
|f-f_\Om|
\leq
\Bigl(
\int_\Om
|f-f_\Om|^{\frac d{d-1}}
\Bigr)^{\frac{d-1}d}
\lesssim
\var_\Om f.
\]

For a general $f\in L^1_\loc(\Omega)$ and $N\in\mathbb{N}$ truncate \(f_N=\max\{\min\{f_N,N\},-N\}\).
Then for $N>|\median \Omega f|$ we have $\median\Omega{f_N}=\median\Omega f$.
Thus, we can conclude from monotone convergence and the previous bounded case
\begin{align*}
\frac1{\lm\Om^{\frac1d}}
\int_\Omega
|f-\median \Omega f|
&=
\lim_{N\rightarrow\infty}
\frac1{\lm\Om^{\frac1d}}
\int_\Omega
|f_N-\median \Omega{f_N}|
\\
&\lesssim
\lim_{N\rightarrow\infty}
\var_\Om(f_N)
=
\var_\Om f
.
\end{align*}
In particular $f\in L^1(\Omega)$, and we can conclude the proof from \zcref{eq:fomegavsinf}.
\end{proof}

\begin{lemma}\label{lem:maximal function and upper representative}
	Let $f\in\BV_{\loc}(\Om)$.
	Then $\M f(x)\ge f^{\vee}(x)$ for
	every $x\in \Om\setminus S_f$.
\end{lemma}
\begin{proof}
	Recall that always $\M f\ge f^*$, and $f^*(x)=f^{\vee}(x)$ for every
	$x\in \Om\setminus S_f$ by \zcref{eq:representatives outside jump set}.
\end{proof}

\begin{proposition}\label{prop:weak type estimate inifinity}
	Let $f\in L^1_{\loc}(\Om)$ with \(\var_\Omega f<\infty\).
	Then
	\[
	\capa_1(\{x\in \Om\colon \M|f|(x)=\infty\})=0.
	\]
\end{proposition}
\begin{proof}
For $x\in\Om$ we call a sequence $r_1,r_2,\ldots\in\R$ a \emph{maximizing sequences} if $f_{B(x,r_i)}\rightarrow\M|f|(x)$,
We group the set of all $x\in\Om$ with $\M|f|(x)=\infty$ into three parts, based on their maximizing sequences.

First consider those $x$ for which all maximizing sequences satisfy $r_i\rightarrow0$.
Then $\M|f|(x)=|f|^*(x)$ and
we know that $|f|^*(x)<\infty$ for $\capa_1$-a.e.\ $x\in \Om$, due to 
\zcref{eq:null sets of capa and Hausdorff,eq:representatives outside jump set,eq:representatives in jump set,eq:Sf and Jf}.

Next we consider those $x$ for which there exists an $0<r<\infty$ and a maximizing sequence with $r_i\rightarrow r$.
Then $B(x,r)\subset\Om$ and
\(
\M |f|(x)=|f|_{B(x,r)}
\)
and by \(\var_{B(x,r)} f\le\var_\Omega f<\infty\) and \zcref{johnpoincare} we have $f\in L^1(B(x,r))$.

For all remaining $x$ there exists a maximizing sequence with $r_i\rightarrow\infty$. 
Now necessarily $\Om=\R^d$. 
Then by e.g.\ \cite[Theorem 3.47]{AFP} there exists a $c\in\R$ such that $f-c\in L^{d/(d-1)}(\R^d)$.
That means
\[
\vint_{B(x,r)}|f-c|
\le \left(\vint_{B(x,r)}|f-c|^{d/(d-1)}\right)^{(d-1)/d}\to 0
\]
as $r\to \infty$ and therefore $\M|f|(x)=c<\infty$.
\end{proof}

\begin{lemma}\label{lem:uniform convergence of u}\cite[Lemma 3.9]{panubvsobolev}
	For any dimension $d\in\mathbb{N}$ there exists a $C_1\geq0$ such that the following holds:
	Let $f\in\BV_{\loc}(\Om)$, let $G\subset \Om$,
	and let $\eps>0$. Then there exists an open set
	$U\supset G$ such that $\capa_1(U)\le C_1\capa_1(G)+\eps$ and
	\[
	\frac{1}{\mathcal L(B(x,r))}\int_{B(x,r)\cap G}|f|
	\to 0\quad\textrm{as }r\to 0
	\]
	locally uniformly for $x\in \Om\setminus U$.
\end{lemma}

Next, we find an exceptional set outside of which several regularity properties hold, which will be used repeatedly.

\begin{lemma}
\label{lem_exceptional}
Let \(f\in L^1_\loc(\Omega)\) with \(\var_\Omega f<\infty\).
Then for every $\varepsilon>0$ exists a set $G\supset(S_f\setminus J_f)$ such that
$\restrict{\M f}{\Om\setminus G}$ is finite and $\restrict{f^{\vee}}{\Om\setminus G}$ is finite and upper semicontinuous,
and an open set $U\supset G$ with $\capa_1(U)<\varepsilon$
such that
\begin{equation}\label{eq:uniform convergence G}
	\frac{1}{\mathcal L(B(x,r))}\int_{B(x,r)\cap G}|f|
	\to 0\quad\textrm{as }r\to 0
\end{equation}
locally uniformly for $x\in \Om\setminus U$.
\end{lemma}

\begin{proof}
	By \zcref{thm:quasisemicontinuity} there is a set $G$ with $\capa_1(G)<\varepsilon/(2C_1)$ such that
	$\restrict{f^{\vee}}{\Om\setminus G}$ is finite and upper semicontinuous.
	By \zcref{prop:weak type estimate inifinity}, the set where $\M|f|$ is infinite has zero capacity,
	and by \zcref{eq:Sf and Jf,eq:null sets of capa and Hausdorff} we have $\capa_1(S_f\setminus J_f)=\sm{S_f\setminus J_f}=0$.
	That means we can include both previous sets in $G$ without increasing its capacity.
	Finally, we find $U\supset G$ by \zcref{lem:uniform convergence of u} as desired.
\end{proof}

Next, we show that $f^*$ is 1-quasicontinuous outside of $J_f$.

\begin{proposition}\label{prop:quasicontinuity outside jumps}
Let \(f\in L^1_\loc(\Omega)\) with \(\var_\Omega f<\infty\).
Then $f^*$ is $1$-quasicontinuous outside of $J_f$, that is, for every $\eps>0$
there exists an open set $G\subset \Om$ such that $\capa_1(G)<\eps$ and
$\restrict{f^*}{\Om\setminus G}$ is finite and continuous at every
$x\in \Om\setminus (J_f\cup G)$.
\end{proposition}
\begin{proof}
We apply \zcref{thm:quasisemicontinuity}, and we note that by
\zcref{eq:Sf and Jf,eq:null sets of capa and Hausdorff} we have
$\capa_1(S_f\setminus J_f)=\sm{S_f\setminus J_f}=0$, so that $S_f\setminus J_f$ can be included
in the exceptional set $G$. Recalling also \zcref{eq:representatives outside jump set}, the 
result follows.
\end{proof}

Finally, we can show the same for the maximal function.

\begin{theorem}\label{thm:quasicontinuity of noncentered maximal function}
Let \(f\in L^1_\loc(\Omega)\) with \(\var_\Omega f<\infty\).
Then $\M f$ is $1$-quasicontinuous outside of $J_f$.
If $|\jump f|(\Om)=0$, then $\M f$ is $1$-quasicontinuous.
\end{theorem}

Note that we have to exclude $J_f$ because it may be a set with nonzero capacity
in which $\M f$ is discontinuous, as can be seen for example for $f=\ind{B(0,1)}$.

\begin{proof}
	For $\eps>0$ take $G,U$ from \zcref{lem_exceptional}.
	Since $\M f$ is finite and lower semicontinuous in $\Om\setminus U$, it is
	sufficient to prove the upper semicontinuity of $\restrictl{\M f}{\Om\setminus U}$ on $\Om\setminus(U\cup J_f)$.
	Fix \(x\in\Om\setminus(U\cup J_f)\).
	Take a sequence $x_i\to x$, $x_i\in \Om\setminus U$, such that
	\[
	\lim_{i\to\infty}\M  f(x_i)=\limsup_{\Om\setminus U\ni y\to x}\M  f(y).
	\]
	(At this stage we cannot exclude the possibility that the $\limsup$ is $\infty$.)
	Now we only need to show 
	$\M f(x)\ge \lim_{i\to\infty}\M f(x_i)$.
	
	We find radii $r_i>0$ such that 
	\begin{equation}\label{eq:choice of almost optimal balls}
		\lim_{i\to\infty}\M f(x_i)=\lim_{i\to\infty}\,f_{B(x_i,r_i)}
	\end{equation}
	Now we consider three cases.\\
	
	\textbf{Case 1.} Suppose that by passing to a subsequence (not relabeled),
	we have $r_i\to 0$. Fix $\delta>0$.
	By \zcref{thm:quasisemicontinuity} the map $\restrictl{f^{\vee}}{\Om\setminus G}$ is upper semicontinuous.
	That means for some
	$r>0$ we have $B(x,r)\subset\Om$ and
	\[
		f^{\vee}(x)\ge \sup_{B(x,r)\setminus G}f^{\vee}-\delta.
	\]
	Note that for sufficiently large $k\in\N$, we have
	$B(x_i,r_i)\subset B(x,r)$.
	Note also that $f\in L^1(B)$ for every open ball $B\subset \Om$, due to 
	$\var_\Omega f<\infty$ and the Poincar\'e inequality (see \zcref{johnpoincare})
	which justifies subtracting integrals over (subsets of) balls below.
	Using \zcref{lem:maximal function and upper representative}
	(recall that \(x\notin U\cup J_f\supset S_f\)),
	for large $i\in\N$ we get 
	\begin{align*}
		\M  f(x)
		&\ge f^{\vee}(x)
		\ge \sup_{B(x,r)\setminus G}f^{\vee}-\delta\\
		&\ge \frac{1}{\mathcal L(B(x_i,r_i))}\int_{B(x_i,r_i)\setminus G}f-\delta\\
		&= \frac{1}{\mathcal L(B(x_i,r_i))}\int_{B(x_i,r_i)}f-
		\frac{1}{\mathcal L(B(x_i,r_i))}\int_{B(x_i,r_i)\cap G}f-\delta
		\\
		&\rightarrow
		\lim_{i\to\infty}\M f(x_i)-\delta
		\qquad\text{ by \zcref{eq:uniform convergence G,eq:choice of almost optimal balls}}.
	\end{align*}
	Letting $\delta\to 0$, we obtain the desired inequality.
	
	\textbf{Case 2.}
	The second alternative is that by passing to a subsequence (not relabeled),
	we have $r_i\to r\in (0,\infty)$.
	Take $k\in \N$ so large that $|x_i-x|\le \tfrac 1{10} \inf_{m\ge k}r_m$ for all $i\ge k$ and denote
	\[
	\Omega_x
	=
	B(x,r)
	\cup
	\bigcup_{i\ge k}B(x_i,r_i)
	.
	\]
	Then
	\(
	\Omega_x
	\subset
	\Omega
	\)
	is a John-domain: The curve $\gamma$ witnessing this for $y\in\Omega_x$ can be taken to be a straight line from the center point $x$ to $x_k$ and then a straight line from $x_k$ to $y$.
	By \zcref{johnpoincare} we can conclude $f\in L^1(\Omega_x)$.
	Since
	\(
	\lm{
	B(x_i,r_i)
	\Delta
	B(x,r)
	}
	\)
	tends to $0$ as $i\rightarrow0$ this also means
	\[
	\lim_{r\rightarrow\infty}
	\int_{
	B(x_i,r_i)
	\Delta
	B(x,r)
	}
	f
	=0
	.
	\]
	Since $\lm{B(x_i,r_i)}\rightarrow\lm{B(x,r)}>0$ we can conclude
	\[
	\M  f(x)\ge \vint_{B(x,r)}f
	=\lim_{i\to\infty}\,\vint_{B(x_i,r_i)}f
	=\lim_{i\to\infty}\M  f(x_i)
	.
	\]
	
	\textbf{Case 3.} Finally, we have the possibility that passing to a subsequence (not relabeled),
	we have $r_i\to \infty$. Note that now necessarily $\Om=\R^d$.
	For $i$ sufficiently large that $|x_i-x|<r_i$, for $d\ge 2$ by \zcref{johnpoincare} we have
	\begin{align*}
	|f_{B(x,r_i)}-f_{B(x_i,r_i)}|
	&\le
	|f_{B(x,r_i)}-f_{B(x,2r_i)}|
	+
	|f_{B(x_i,r_i)}-f_{B(x,2r_i)}|
	\\
	&\le
	\frac{
	2^{d+1}r_i
	}{
	\lm{B(x,2r_i)}
	}|Df|(B(x,2r_i))
	\to 0
	\end{align*}
	as $i\to \infty$ since $|Df|(\R^d)<\infty$.
	For $d=1$, we can estimate
	\begin{align*}
		|f_{B(x,r_i)}-f_{B(x_i,r_i)}|
		&\le
		\frac{
		\lm{B(x,r_i)\Delta B(x_i,r_i)}
		}{2r_i}
		\|f\|_{L^{\infty}(\R)}
		\to 0.
	\end{align*}
	Then
	\[
	\M f(x)\ge \limsup_{i\to\infty}f_{B(x,r_i)}
	=
	\limsup_{i\to\infty}
	f_{B(x_i,r_i)}
	=
	\lim_{i\to\infty}\M f(x_i).
	\]
	This completes the proof.	
	
	Finally, if $|\jump f|(\Om)=0$, then also the set $J_f$ can be included in $G$, giving the result.
\end{proof}

\subsection{Lusin property}

Recall the definition \zcref{eq:auxiliary maximal operator}.

\begin{proposition}\label{prop:Lipschitz continuity}
	Let \(f\in L^1_\loc(\Omega)\) with \(\var_\Omega f<\infty\), and let $R>0$.
	Then $\Mg {R} f$ is \(2\var_\Omega f/\lm{B(0,R)}\)-Lipschitz on $\{x:\dist(x,\R^d\setminus\Om)\geq R\}$ with respect to shortest path distance.
\end{proposition}

\begin{proof}
	For any $r>0$ and $x\in\Omega$ define \(r(x)=\min\{r,\dist(x,\Omega^\complement)\}\) and
	\[
	g_r(x)
	=
	\vint_{B(x,r(x))}
	|f|
	=
	\int_{B(0,1)}
	|f|(x+r(x)y)
	\intd y
	.
	\]
	Note that the map $r(\cdot)$ is 1-Lipschitz.
	Let $\nu\in\Sph^{d-1}$.
	Then
	\begin{align*}
	|\nabla g_r(x)|
	&=
	\Bigl|
	\int_{B(0,1)}
	\nabla(\cdot+r(\cdot)y)(x)
	\cdot
	\nabla|f|(x+r(x)y)
	\intd y
	\Bigr|
	\\
	&=
	\Bigl|
	\int_{B(0,1)}
	(\id+\nabla r(x)\otimes y)
	\cdot
	\nabla|f|(x+r(x)y)
	\intd y
	\Bigr|
	\\
	&\leq2
	\int_{B(0,1)}
	|\nabla f|(x+r(x)y)
	\intd y
	\\
	&\leq
	2\frac{\var_\Omega f}{\lm{B(x,r(x))}}
	\end{align*}
	That means $\Mg Rf$ is a supremum of maps that are \(2\var_\Omega f/\lm{B(0,R)}\)-Lipschitz, and thus so is itself.
\end{proof}

\begin{lemma}\label{lem:1d abs cont}\cite[Lemma 5.2]{panubvsobolev}
	Let $V\subset\R$ be open and let $f\in\BV_{\loc}(V)$.
	If $N\subset V\setminus S_f$ with $|Df|(N)=0$, then
	\[
	\mathcal L^1(f^*(N))=0.
	\]
\end{lemma}

The \emph{Lusin property} for a function
$v$ defined on $V\subset\R$ states that
\[
	\textrm{if }N\subset V\textrm{ with }\mathcal L^1(N)=0,\ \textrm{ then }
	\mathcal L^1(v(N))=0.
\]

The \emph{measure-theoretic boundary} $\mb E$ of a set $E\subset \R^d$ is defined as the set
of points where
\[
\limsup_{r\to 0}
\frac{
\lm{B(x,r)\cap E}
}{
\lm{B(x,r)}
}>0
\quad\textrm{and}\quad
\limsup_{r\to 0}
\frac{
\lm{B(x,r)\setminus E}
}{
\lm{B(x,r)}
}>0.
\]
By the coarea formula for BV functions, see e.g.\ \cite[Theorems 3.40,3.59,3.61]{AFP}, for any
$f\in \BV_{\loc}(\Om)$ we have
\[
|Df|(\Om)=\int_{-\infty}^\infty\sm{\Om\cap\mb{\{x\in\Om:f(x)>t\}}}\intd t
.
\]

\begin{lemma}\label{lem:var meas}
Suppose $W\subset \R$ is open and $u,v\in \BV_{\loc}(W)$.
Suppose $A\subset W$ is such that $u$ is continuous and zero in all points in $A$.
Then
\[
|Du|(A)=0.
\]
\end{lemma}
\begin{proof}
Note that $A\cap \partial^*\{u>t\}=\emptyset$ for all $t\neq 0$. Now the result follows from
the coarea formula.
\end{proof}

\begin{lemma}
\label{eq:maximal function is M R}
Let $G,U$ be as in \zcref{lem_exceptional}.
Then for every $x_0\in\Om\setminus U$ with $\M f(x_0)>f^\vee(x_0)$ there exists an $R>0$ such that for all $x\in B(x_0,R)\setminus U$ we have $\M f(x)=\Mg Rf(x)$.
\end{lemma}

\begin{proof}
Abbreviate $\alpha\coloneqq\M f(x_0)\in (-\infty,\infty)$ and $\delta\coloneqq\alpha-f^{\vee}(x_0)>0$.
By upper semicontinuity of $f^{\vee}$ on $\Om\setminus G$,
there exists an $R_1>0$ such that for all $x\in B(x_0,R_1)\setminus G$
we have
\[
	f^{\vee}(x)<\alpha-\frac{3\delta}{4}
	.
\]
By lower semicontinuity of the maximal function,
there exists an $R_2>0$ such that
$\M f(x)>\alpha-\delta/4$ for all $x\in B(x_0,R_2)$.
By the choice of $U$ there exists an $R_3>0$
such that for all $x\in B(x_0,R_3)\setminus U$ we have
\[
	\frac{1}{\mathcal L(B(x,s))}\int_{G\cap B(x,s)}|f|
	\le\frac{\delta}{2}
	\qquad\textrm{for all }0<s<R_3.
\]
Set $R=\min\{R_1,R_2,R_3\}/2$
and let $x\in B(x_0,R)\setminus U$
and $r\in (0,R)$. Then we can conclude
\begin{align*}
	\frac{1}{\mathcal L(B(x,r))}\int_{B(x,r)}f
	&= \frac{1}{\mathcal L(B(x,r))}\int_{B(x,r)\cap G}f
	+\frac{1}{\mathcal L(B(x,r))}\int_{B(x,r)\setminus G}f\\
	&\le \frac{1}{\mathcal L(B(x,r))}\int_{B(x,r)\cap G}f
	+\alpha-\frac{3\delta}{4}\\
	&\le\frac{\delta}{2}+\alpha-\frac{3\delta}{4}
	=\alpha-\frac{\delta}{4}.
\end{align*}
On the other hand, we had $\M f(x)>\alpha-\delta/4$ for all
$x\in B(x_0,R)$.
That means
\(
	\M f(x)=\Mg {R}f(x)
.
\)
\end{proof}

\begin{theorem}\label{thm:continuity on lines}
Let
\(f\in L^1_\loc(\Omega)\) with \(\var_\Omega f<\infty\)
such that at $|\cantor f|$-a.e.\ $x\in \Om$, we have $\M f(x)>f^*(x)$.
Then $\M f$ has the Lusin property
on almost every line parallel to a coordinate axis.
\end{theorem}

\begin{proof}
Recall the notation and results from \zcref{subsec:one dimensional sections} and
that $\pi$ denotes the orthogonal projection onto $\R^{d-1}$,
and recall \zcref{lem:capacity and Hausdorff measure} which says
\(
2\smo\circ\pi\le \capa_1
.
\)
That means that for  $\smo$-almost every $z\in\pi(\Om)$ there exist sets $G,U$
from \zcref{lem_exceptional} with $z\in \pi(\Om)\setminus \pi(U)$.
In other words, for almost every line parallel to the $d$th coordinate axis exist $G,U$ which do not intersect said line.
In addition, the following hold for almost every $z\in \pi(\Om)$, i.e.\ for almost every line parallel to the $d$th coordinate axis:
We have $f_z\in\BV_{\loc}(\Om_z)$,
the set $(J_f)_z=J_{f_z}$ is at most countable, and \zcref{eq:sections and jump sets} holds;
recall \zcref{subsec:one dimensional sections}.
Denote
\[
H_f\coloneqq\{\M f>f^{\vee}\}.
\]
By \zcref{eq:representatives outside jump set}, for $|\cantor f|$-a.e.\ $x\in\Om$ we have $f^{\vee}(x)=f^*(x)$.
Then by assumption $|\cantor f|(\Om\setminus H_f)=0$, and so
due to \zcref{eq:sections and jump sets 3} we have
\begin{equation}
\label{eq:cantorzzero}
|\cantor{(f_z)}|(\Om_z\setminus (H_f)_z)=0
\end{equation}
for $\lmlo$-a.e.\ $z\in\pi(\Om)$.
By symmetry all of the above also holds for almost every line parallel to any other
coordinate axis.

For simplicity we only consider a line $\{(z,t):t\in \R\}$ parallel to the $d$th coordinate axis.
For a set $N\subset \R$ with zero $1$-dimensional Lebesgue measure, we split
\begin{align*}
\M f(\{z\}\times N)
&=
\M f((\{z\}\times N)\cap H_f)
\cup
\M f((\{z\}\times N)\cap J_f)
\\
&\qquad\cup
\M f((\{z\}\times N)\setminus(H_f\cup J_f).
\end{align*}
The first set has zero one-dimensional Lebesgue measure since $\M f$ is locally Lipschitz on $H_f\setminus U$ due to
\zcref{eq:maximal function is M R,prop:Lipschitz continuity}.
So does the second set since we assumed $(\{z\}\times \Om_z)\cap J_f$ is at most countable.
For the third set, since $U\supset S_f\setminus J_f$ and
since we assumed \zcref{eq:sections and jump sets} to hold, we have
\begin{align}
	\notag
	\M f((\{z\}\times N)\setminus (H_f\cup J_f))
	&=
	\M f((\{z\}\times N)\setminus (H_f\cup S_f))
	\\
	\notag
	&= f^*((\{z\}\times N)\setminus (H_f\cup S_f))
	\\
	\label{eq:Mf union}
	&\subset f_z^*(N\setminus (S_{f_z}\cup (H_f)_z)
	.
\end{align}
By our assumptions
\begin{align*}
|\cantor{(f_z)}| (\Om_z\setminus (H_f)_z))
&=0
,&
|\ac{(f_z)}| (N)
&=
0
,&
|\jump{(f_z)}| (\Om_z\setminus S_{f_z})
&=
0
\end{align*}
so that by \zcref{lem:1d abs cont} the third set in \zcref{eq:Mf union} has zero $\lmio$-measure.
In conclusion, $\lmi{\M f(z,N)}=0$, and so $\M f$ has the Lusin property
on almost every line parallel to a coordinate axis.
\end{proof}

\begin{theorem}\label{thm:Cantor removed}
Suppose 
\(f\in L^1_\loc(\Omega)\) with \(\var_\Omega f<\infty\),
such that at $|\cantor f|$-a.e.\ $x\in \Om$, we have $\M f(x)>f^*(x)$.
Moreover, suppose that $\M  f\in \BV_{\loc}(\Om)$.
Then $|D^c \M f|(\Om)=0$.
\end{theorem}
\begin{proof}
As in the beginning of the proof of \zcref{thm:continuity on lines}, for $\lmlo$-a.e.\ line parallel to a coordinate axis,
indexed by $z\in\pi(\Om)$,
there exist $G,U$ from \zcref{lem_exceptional} which do not intersect that line.
Moreover,
\(f_z\in\BV_{\loc}(\Om_z)\)
and \zcref{eq:cantorzzero} holds.
In addition, $f^*$ and $\M f$ are continuous outside of $J_f$ on almost every line
by \zcref{prop:quasicontinuity outside jumps,thm:quasicontinuity of noncentered maximal function}.

For simplicity we only consider a line $\{(z,t):t\in \R\}$ parallel to the $d$th coordinate axis.
We consider a set $N\subset \Om_z$ with zero $1$-dimensional Lebesgue measure.
By \zcref{lem:var meas,eq:cantorzzero} we have
\begin{align*}
|\cantor{[(\M f)_z]}|(\Om_z\setminus (H_f)_z)
&=|\cantor{[(\M f)_z]}|(\Om_z\setminus ((H_f)_z\cup (J_f)_z))\\
&=|\cantor{[f_z]}|(\Om_z\setminus ((H_f)_z\cup (J_f)_z))
=0
.
\end{align*}
By \zcref{eq:maximal function is M R,prop:Lipschitz continuity} the map $\restrict{ \M f}{\Om\setminus U}$
is locally Lipschitz and thus
\[
|D[(\M f)_z]|(N\cap (H_f)_z)=0
.
\]
That means
\(
|\cantor{[(\M f)_z]}|(N)=0
.
\)
Since $\cantor{[(\M f)_z]}$ is carried by a null set, in fact $|\cantor{[(\M f)_z]}|(\Om_z)=0$.
Using again \zcref{eq:sections and jump sets 3}, we can conclude $|\cantor{\M f}|(\Om)=0$.
\end{proof}

\begin{proposition}\label{prop:jump halved}
	Suppose 
	\(f\in L^1_\loc(\Omega)\) with \(\var_\Omega f<\infty\),
	and suppose that $\M  f\in \BV_{\loc}(\Om)$.
	Then $|\jump \M f|\le \tfrac 12 |\jump f|$.
\end{proposition}
\begin{proof}
	Take sets $G,U$ from \zcref{lem_exceptional}.
	Given $G$, by applying \zcref{lem:uniform convergence of u} to a constant function, we obtain
	\begin{equation}
	\label{eq_Gfillsball}
		\frac{\mathcal L(G\cap B(x,r))}{\mathcal L(B(x,r))}
		\to 0\quad\textrm{as }r\to 0
	\end{equation}
	locally uniformly for $x\in \Om\setminus U$ after slightly enlarging $U$ while still respecting $\capa_1(U)<\varepsilon$.
	By \zcref{thm:quasicontinuity of noncentered maximal function}, we can also assume that
	$\restrictl{\M f}{\Om\setminus G}$ is continuous at every point $x\in \Om\setminus(G\cup J_f)$.

	In particular, for any $x\in \Om\setminus (U\cup J_f)$ there exists a $c\in\mathbb{R}$ for which
\(\restrictl{(\M f-c)}{B(x,r)\cap\Om\setminus G}\rightarrow0\) uniformly as $r\rightarrow0$.
By \zcref{eq_Gfillsball} the set $G$ where this convergence happens eventually fills all of $B(x,r)$.
This is not compatible with $x\in J_{(\M f)_z}$, which, due to \zcref{eq_approximatejumpdefinition}, would require the blow-up of $\M f$ to converge to two distinct values on two halves of the ball.
	Since $\capa_1(U)$ can be made arbitrarily small, we can conclude
	$\capa_1(J_{\M f}\setminus J_f)=0$ 
which by \zcref{eq:null sets of capa and Hausdorff} implies
$\sm{J_{\M f}\setminus J_f}=0$.

Suppose $x\in J_{\M f}\cap J_f$. We have
\[
\M f(x)\ge f^*(x)=\frac{f^{\wedge}(x)+f^{\vee}(x)}{2},
\]
and by the lower semicontinuity of $\M f$ also
\[
(\M f)^{\wedge}(x)\ge \frac{f^{\wedge}(x)+f^{\vee}(x)}{2}.
\]
If $(\M f)^\vee(x)\le f^{\vee}(x)$, then we get
\begin{equation}\label{eq:less than half}
(\M f)^\vee(x)-(\M f)^\wedge(x)\le \tfrac 12 (f^\vee(x)-f^\wedge(x)).
\end{equation}
Otherwise, for $\kappa\coloneqq(\M f)^\vee(x)-f^{\vee}(x)>0$,
for some half-space $H$ with $x$ on its boundary
and for all $0<\delta<1/2$ we have
\begin{equation}\label{eq:half space H}
\lim_{r\to 0}\frac{\mathcal L(B(x,r)\cap H\cap \{\M f>f^{\vee}(x)+(1-\delta)\kappa\})}{\mathcal L(B(x,r)\cap H)}
=
1
.
\end{equation}
Fix such $\delta$.
We want to show that there is an $R>0$ and sequence $H\ni y_i\to x$ with 
\begin{align}
\label{eq:Rsequence}
\M f(y_i)
&>
f^{\vee}(x)+(1-\delta)\kappa
,&
\M f(y_i)
&=
\Mg R f(y_i)
\end{align}
Take $R>0$ sufficiently small so that for all $0<r\le R$,
\begin{equation}\label{eq:45 d}
\frac{1}{\delta\kappa}\vint_{B(x,3r)}(f-f^{\vee}(x)-(1-2\delta)\kappa)_+<\frac{1}{45^d}.
\end{equation}
Fixed such $r$, let
 $\B$ be the set of balls $B$ with center in $B(x,r)$ with radius at most $R$ and \(f_B>f^\vee(x)+(1-\delta)\kappa\).
Then
\[
B(x,r)\cap\{\Ml {R}f >f^\vee(x)+(1-\delta)\kappa\}\subset\bigcup\B.
\]
By the $5$-covering lemma we find a disjoint subcover of $\B$, denoted by $\{B_j\}_{j}$,
such that $\bigcup_{j}5B_j$ still contains $B(x,r)\cap\{\M f>f^\vee(x)+(1-\delta)\kappa\}$.
If there were a ball $B_j$ with radius $r_j\ge 2r$, we would have
\begin{align*}
&
\frac{1}{\delta\kappa}\vint_{B(x,3r_j)}(f-f^{\vee}(x)-(1-2\delta)\kappa)_+
\\	
&\qquad\ge \frac{1}{2^d\delta\kappa}\vint_{B_j}(f-f^{\vee}(x)-(1-2\delta)\kappa)_+
\ge \frac{1}{2^d},
\end{align*}
contradicting \zcref{eq:45 d}.
Thus all balls $B_j$ have radius at most $2r$, and so
\begin{align*}
&
\mathcal L(B(x,r)\cap\{\Ml {R}f>f^\vee(x)+(1-\delta)\kappa\})
\le
5^d\sum_j\mathcal L(B_j)
\\
&\qquad\le
5^d\frac{1}{\delta\kappa}\sum_j\int_{B_j}(f-f_{B_j}-(1-2\delta)\kappa)_+
\\
&\qquad\le
5^d\frac{1}{\delta\kappa}\int_{B(x,3r)}(f-f_{B_j}-(1-2\delta)\kappa)_+
<\frac{1}{3}\mathcal L(B(x,r)).
\end{align*}
Due to \zcref{eq:half space H}
we obtain \zcref{eq:Rsequence}.
By the continuity of $\Mg Rf$ we can conclude
\[
\liminf_{y\rightarrow x}
\M f(y)
\geq
\liminf_{y\rightarrow x}
\Mg Rf(y)
= \lim_{i\to\infty}\Mg R f(y_i)
\ge f^{\vee}(x)+(1-\delta)\kappa.
\]
Letting $\delta\to0$ we obtain
\[
(\M f)^\wedge(x)
\ge
\liminf_{y\rightarrow x}
\M f(y)
\ge f^{\vee}(x)+\kappa
= (\M f)^\vee(x)
\]
which means \zcref{eq:less than half} holds also in this case as the left hand side is zero.
Finally, we can conclude the proof from \zcref{eq:representatives in jump set}
and \zcref{eq:jump part representation}.
\end{proof}

\begin{proof}[Proof of \zcref{thm:main two}]
	The estimate $|\jump{\M f}|\leq\frac12|\jump f|$ is proved in \zcref{prop:jump halved}.
	
	The second claim follows by combining \zcref{thm:Cantor removed}
	and \zcref{thm:main}.
	
	For the third claim, suppose that $|\jump f|(\Om)=0$.
	Now by the first claim, the singular part  $|\jump{\M f}|$
	vanishes and so $\M f\in W^{1,1}_{\loc}(\Om)$.
	Consider the $d$:th coordinate direction; the other directions are analogous.
	Since $\M f\in W^{1,1}_{\loc}(\Om)$, the function $\M f(z,\cdot)$ is in
	the class $W^{1,1}_{\loc}(\Om_z)$ for almost every $z\in \pi(\Om)$.
	By \zcref{thm:quasicontinuity of noncentered maximal function},
	\zcref{lem:capacity and Hausdorff measure} and \zcref{eq:jump part representation},
	$\M f(z,\cdot)$ is continuous for almost every $z\in \pi(\Om)$.
	Thus $\M f(z,\cdot)$ is absolutely continuous for
	almost every $z\in \pi(\Om)$, and in conclusion $\M f$ is ACL.
\end{proof}

\printbibliography

\end{document}